\documentclass[11pt,reqno]{amsart}

\usepackage[T1]{fontenc}
\usepackage{lmodern}
\usepackage{microtype}
\usepackage{mathtools,amssymb}
\usepackage{booktabs}
\usepackage{tabularx,array}
\usepackage{float}
\usepackage{enumitem}
\usepackage{xcolor}
\usepackage[colorlinks=true,linkcolor=blue!55!black,citecolor=blue!55!black,
  urlcolor=blue!55!black]{hyperref}
\usepackage[capitalize,noabbrev]{cleveref}

\allowdisplaybreaks
\numberwithin{equation}{section}
\newtheorem{theorem}{Theorem}[section]
\newtheorem{proposition}[theorem]{Proposition}
\newtheorem{lemma}[theorem]{Lemma}

\theoremstyle{definition}

\newtheorem{remark}[theorem]{Remark}

\newcommand{\R}{\mathbb R}
\newcommand{\C}{\mathbb C}
\newcommand{\Q}{\mathbb Q}
\newcommand{\Z}{\mathbb Z}
\newcommand{\T}{\mathbb T}
\newcommand{\cS}{\mathcal S}
\newcommand{\cZ}{\mathcal Z}
\newcommand{\cZtwo}{\mathcal Z_2}
\newcommand{\ip}[2]{\langle #1,#2\rangle}
\newcommand{\norm}[1]{\lVert #1\rVert}
\newcommand{\abs}[1]{\lvert #1\rvert}
\newcommand{\e}{\mathrm e}
\newcommand{\ii}{\mathrm i}
\newcommand{\Frob}{\mathrm F}

\title[An intrinsically subcritical four-point counterexample]
  {An intrinsically subcritical four-point counterexample\\
   to the HRT conjecture}

\author{Vignon Oussa}
\address{Department of Mathematics, Bridgewater State University,
Bridgewater, Massachusetts, USA}
\email{voussa@bridgew.edu}
\date{August 6, 2026}


\begin{document}

\begin{abstract}
Building on the vector-Zak and cohomological framework developed by
Faulhuber, Petersen, van Velthoven, and Voigtlaender in their twelve-point
counterexample, we give a computer-assisted four-point counterexample with a
nonzero complex-valued Schwartz window.  Every symplectic triangle determinant
of the explicit configuration has absolute value below one, placing it in the
intrinsically subcritical regime.
\end{abstract}

\maketitle

\section{Introduction}

For $z=(x,\omega)\in\R^2$, let the Weyl shift $\rho(z)$ on $L^2(\R)$ be
\begin{equation}\label{eq:weyl-definition}
  (\rho(x,\omega)f)(t)
  =\e^{2\pi\ii\omega(t-x/2)}f(t-x).
\end{equation}
The Heil--Ramanathan--Topiwala conjecture asserted that, for every nonzero
$f\in L^2(\R)$ and every finite set of distinct phase-space points
$z_1,\dots,z_n$, the functions $\rho(z_1)f,\dots,\rho(z_n)f$ are linearly
independent \cite{HRT}.  The Weyl normalization differs from the usual
time--frequency shifts only by unimodular constants.

In 2026, Faulhuber, Petersen, van Velthoven, and Voigtlaender disproved the
assertion by constructing twelve dependent time--frequency shifts of a
nonzero Schwartz function \cite{FPPV2026}.  Their work introduced the decisive
combination of ingredients used here: a two-component vector Zak model
\cite{ZibZee97}, an irrational
cubic torus translation, a smooth nowhere-zero reference section, a
multiplier-winding argument, and smooth cohomological reconstruction.  The
present paper retains that architecture but replaces their eleven-term Weyl
polynomial by the three-term half-lattice polynomial
\begin{equation}\label{eq:lattice-polynomial-intro}
  I+\frac35\rho(1,0)+\frac35\rho(0,1/2).
\end{equation}
Its invariant line is obtained from a rigorous sixteen-step domination
certificate, and a two-gauge half-plane argument controls the multiplier
winding.  Together these mechanisms reduce the configuration cardinality
from twelve points to four.

Set
\begin{equation}\label{eq:parameters}
 \vartheta=\sqrt[3]{2},\qquad
 \alpha=\vartheta-1,\qquad
 \beta=\vartheta^2-1,\qquad
 \zeta=(\alpha,\beta/2).
\end{equation}

\begin{theorem}[Main theorem]\label{thm:main}
There exist a nonzero complex-valued function $g\in\cS(\R)$ and a number
$\lambda\in\C\setminus\{0\}$ such that
\begin{equation}\label{eq:eigen-relation}
 \left[I+\frac35\rho(1,0)+\frac35\rho(0,1/2)\right]
 \rho(\zeta)g=\lambda g.
\end{equation}
Equivalently,
\begin{align}\label{eq:four-term-relation}
 -\lambda g
 &+\rho(\alpha,\beta/2)g
 +\frac35\e^{-\pi\ii\beta/2}\rho(1+\alpha,\beta/2)g \notag\\
 &+\frac35\e^{\pi\ii\alpha/2}
       \rho\bigl(\alpha,(1+\beta)/2\bigr)g=0.
\end{align}
The four points in \eqref{eq:four-term-relation} are distinct, no three are
collinear, and every absolute symplectic triangle determinant is strictly
between zero and one.
\end{theorem}

The cardinality is minimal: finite systems with at most three distinct points
are known to be independent \cite{HRT,heil2006linear}.
Linnell's theorem covers configurations contained in translates of discrete
subgroups \cite{Linnell}; Demeter proved the relevant $(1,3)$ and
two-parallel-line cases for Schwartz windows \cite{demeter2010linear}; and
Demeter and Zaharescu proved the $(2,2)$ case for arbitrary nonzero $L^2$
windows \cite{demeter2012proof}.  Liu proved the $L^2$ assertion for almost
every $(1,3)$ configuration \cite{Liu2019}.

For the mixed arithmetic problem of three lattice points and one rogue point,
Okoudjou and Oussa exclude the maximally irrational unit-covolume case for
Schwartz (more generally, $W_0(\R)$) windows
\cite{OkoudjouOussa2025}; Oussa proves independence in the maximally
irrational supercritical regime of symplectic covolume greater than one
\cite{Oussa2026}.  Our background covolume is $1/2$, while all four triangle
determinants are below one.  Its geometry avoids the lattice, $(1,3)$, and
$(2,2)$ regimes covered by the preceding results.

\subsection*{The marked one-rogue-point dashboard}

Every noncollinear four-point configuration admits a marked presentation
\begin{equation}\label{eq:marked-one-rogue}
 \{0,u,v,q\},\qquad
 L=\Z u+\Z v,\qquad q=a u+b v,\qquad [u,v]\ne0.
\end{equation}
The marking consists of the reference triple $(0,u,v)$, the lattice it
generates, and the remaining, or \emph{rogue}, point $q$.  Its first label is
geometric,
\[
 \delta=\operatorname{covol}(L)=|[u,v]|,
\]
with supercritical, critical, and subcritical regimes according as
$\delta>1$, $\delta=1$, and $\delta<1$.  Its second label is arithmetic,
\[
 \varrho=\dim_{\Q}\operatorname{span}_{\Q}\{1,a,b\}\in\{1,2,3\},
\]
corresponding to rational, mixed, and maximally irrational rogue coordinates.
The covolume label depends on the chosen reference triangle; hence this is a
taxonomy of marked presentations, not a disjoint classification of unmarked
sets.  The rational rank, however, is unchanged by every admissible
re-marking.  For example, if $b\ne0$, then relative to the basis $(u,q)$ the
old point $v$ has coordinates $(-a/b,1/b)$, and multiplication by $b$ is a
$\Q$-linear isomorphism carrying
$\operatorname{span}_{\Q}\{1,-a/b,1/b\}$ onto
$\operatorname{span}_{\Q}\{b,-a,1\}$.  The other two re-markings are the
same calculation with $a$ and with $1-a-b$, respectively.

The current four-point research dashboard is summarized in
\cref{tab:dashboard}; see also \cite{OussaBook2026}.  It is included to locate
the present result and is not used as an input to its proof.
\begin{table}[H]
\caption{Working dashboard for marked four-point one-rogue configurations.
``Partially open'' means that no theorem covers the entire marked cell.}
\label{tab:dashboard}
\small
\begin{tabularx}{\textwidth}{@{}>{\raggedright\arraybackslash}p{0.16\textwidth}
  >{\raggedright\arraybackslash}X
  >{\raggedright\arraybackslash}X
  >{\raggedright\arraybackslash}X@{}}
\toprule
Geometric regime & $\varrho=1$ (rational) &
$\varrho=2$ (mixed) & $\varrho=3$ (maximally irrational)\\
\midrule
$\delta>1$ (supercritical) &
Complete by rational refinement and Linnell's theorem \cite{Linnell}. &
Partially open; the one-dimensional orbit-closure endpoint remains. &
Complete for four points by the large-covolume theorem \cite{Oussa2026}.\\
\addlinespace
$\delta=1$ (critical) &
Complete by rational refinement and Linnell's theorem. &
Complete in the critical four-point theory; the mechanism is winding and
return holonomy. &
Complete in the critical four-point theory; the mechanism uses phase-current
rigidity, locked zeros, adjoint reflection, and small divisors.\\
\addlinespace
$\delta<1$ (subcritical) &
Complete by rational refinement and Linnell's theorem. &
Partially open; no full-cell theorem is presently available. &
The positive HRT assertion is false in this cell by \cref{thm:main}; a full
classification of dependent configurations remains open.\\
\bottomrule
\end{tabularx}
\end{table}

Thus the theorem corroborates the search principle that genuinely new
four-point counterexamples should be sought in the irrational subcritical
cells.  We record this as an organizing heuristic, not as a theorem asserting
that every counterexample must arise by one fixed marking.

The phrase \emph{computer-assisted proof} is used in the validated-numerics
sense \cite{johansson2017arb,moore2009,rump2010}.  No global conclusion is
drawn from floating-point sampling.  Every coefficient and grid-center value
is an outward-rounded ball, and a derivative estimate covers each closed
cell.  The complete programs and a Lean algebraic companion are available in
the electronic supplement; see \cref{sec:formal-scope}.

\section{The configuration and its geometric position}
\label{sec:geometry}

Write $[u,v]=u_1v_2-u_2v_1$ for the standard symplectic form.  The absolute
symplectic determinant of a triangle $(p,q,r)$ is
$\abs{[q-p,r-p]}$, twice its Euclidean area.  Let
\begin{equation}\label{eq:four-points}
 p_0=(0,0),\quad
 p_1=(\alpha,\beta/2),\quad
 p_2=(1+\alpha,\beta/2),\quad
 p_3=\bigl(\alpha,(1+\beta)/2\bigr).
\end{equation}

\begin{proposition}[Exact subcritical geometry]\label{prop:geometry}
The four absolute symplectic triangle determinants of
\eqref{eq:four-points} are
\begin{equation}\label{eq:four-determinants}
 \frac{\beta}{2},\qquad \frac{\alpha}{2},\qquad
 \frac{1+\alpha+\beta}{2},\qquad \frac12.
\end{equation}
Each lies strictly between $0$ and $1$.  Consequently the points are
distinct, no three are collinear, and every triangle determined by them has
Euclidean area strictly less than $1/2$.
\end{proposition}

\begin{proof}
For the ordered triples
\[
 (p_0,p_1,p_2),\quad (p_0,p_1,p_3),\quad
 (p_0,p_2,p_3),\quad (p_1,p_2,p_3),
\]
direct expansion gives
\[
 -\frac\beta2,\qquad \frac\alpha2,\qquad
 \frac{1+\alpha+\beta}{2},\qquad \frac12.
\]
Since $1<\vartheta<4/3$ and $1<\vartheta^2<5/3$, we have
$0<\alpha<1$, $0<\beta<1$, and
$1+\alpha+\beta=\vartheta+\vartheta^2-1<2$.  The rational bounds follow
by cubing their positive endpoints.
\end{proof}

\begin{proposition}[Rational rank]\label{prop:rational-rank}
The numbers $1,\alpha,\beta$ are linearly independent over $\Q$.  The
configuration \eqref{eq:four-points} is not contained in a translate of a
discrete subgroup of $\R^2$.
\end{proposition}

\begin{proof}
A rational relation among $1,\alpha,\beta$ would give a rational polynomial
of degree at most two vanishing at $\vartheta$, contrary to the
irreducibility of $X^3-2$.

Translate a putative discrete subgroup so that it contains all point
differences.  It would contain $(1,0)$, $(0,1/2)$, and
$p_1=\alpha(1,0)+\beta(0,1/2)$.  A discrete subgroup containing the lattice
generated by the first two vectors contains that lattice with finite index;
hence every element has rational coordinates in this basis.  This would make
$\alpha,\beta$ rational, a contradiction.
\end{proof}

For the three possible pairings, the determinants of the paired direction
vectors are
\[
 \frac{\alpha+\beta}{2},\qquad
 \frac{1+\alpha}{2},\qquad
 -\frac{1+\beta}{2},
\]
so no opposite pairing gives parallel lines.  The configuration is therefore
not of type $(2,2)$; it is not of type $(1,3)$ because no three points are
collinear.

In this paper, \emph{subcritical} means that the symplectic covolume of the
background lattice is below one; this convention is stated because density
terminology can reverse the word.  After translating by $-p_1$, the
background points are $0,u,v$, where $u=(1,0)$ and $v=(0,1/2)$, and the
rogue point is
\[
 p_0-p_1=-\alpha u-\beta v.
\]
Thus the marked presentation has $\delta=1/2$ and $\varrho=3$.  The four
possible nondegenerate marking covolumes are precisely the four absolute
determinants in \eqref{eq:four-determinants}; all are below one.  We therefore
use the stronger phrase \emph{intrinsically subcritical}: no admissible
re-marking moves the configuration out of the subcritical row of
\cref{tab:dashboard}.

\section{Exact vector-Zak reduction}
\label{sec:zak}

For $f\in L^2(\R)$, use the scalar Zak transform
\begin{equation}\label{eq:scalar-zak}
 (\cZ f)(x,\omega)=\sum_{k\in\Z}f(x-k)\e^{2\pi\ii k\omega}
\end{equation}
in the usual $L^2$ sense, and define
\begin{equation}\label{eq:vector-zak}
 (\cZtwo f)_{r+1}(x,\omega)
 =2^{-1/2}(\cZ f)\left(x,\frac{\omega+r}{2}\right),
 \qquad r\in\{0,1\}.
\end{equation}
As in \cite[\S3]{FPPV2026}, this is a unitary map from $L^2(\R)$ onto the
space of measurable $F:\R^2\to\C^2$ satisfying
\begin{align}
 F(x+1,\omega)&=U_1(\omega)F(x,\omega),
 &U_1(\omega)&=\e^{\pi\ii\omega}
                   \begin{pmatrix}1&0\\0&-1\end{pmatrix},
 \label{eq:sewing-x}\\
 F(x,\omega+1)&=U_2F(x,\omega),
 &U_2&=\begin{pmatrix}0&1\\1&0\end{pmatrix}.
 \label{eq:sewing-w}
\end{align}
See also \cite{ZibZee97}.  We call such an $F$ a vector-Zak section.

Put
\[
 D=\begin{pmatrix}1&0\\0&-1\end{pmatrix},\qquad
 S=\begin{pmatrix}0&1\\1&0\end{pmatrix}.
\]
For $m,n\in\Z$, direct use of \eqref{eq:scalar-zak} gives
\begin{equation}\label{eq:half-lattice-action}
 \cZtwo\rho(m,n/2)\cZtwo^{-1}
 =L_{m,n}(x,\omega),\quad
 L_{m,n}=\e^{\pi\ii(nx-m\omega+mn/2)}L_{m,n}^{(0)},
\end{equation}
where
\[
 L_{m,n}^{(0)}=
 \begin{cases}
  \operatorname{diag}(1,\e^{-\pi\ii m}),&n\equiv0\pmod2,\\[1mm]
  \begin{pmatrix}0&1\\ \e^{-\pi\ii m}&0\end{pmatrix},
      &n\equiv1\pmod2.
 \end{cases}
\]
In particular, $L_{1,0}=\e^{-\pi\ii\omega}D$ and
$L_{0,1}=\e^{\pi\ii x}S$.

Let
\begin{equation}\label{eq:base-map}
 \tau=(\alpha,\beta),\qquad Tz=z-\tau,\qquad z=(x,\omega).
\end{equation}
The irrational Weyl shift acts by
\begin{equation}\label{eq:irrational-action}
 (\cZtwo\rho(\zeta)\cZtwo^{-1}F)(x,\omega)
 =\eta(x)F(T(x,\omega)),\qquad
 \eta(x)=\e^{\pi\ii\beta(x-\alpha/2)}.
\end{equation}
Consequently,
\begin{equation}\label{eq:zak-cocycle-action}
 \cZtwo\left[I+\frac35\rho(1,0)+\frac35\rho(0,1/2)\right]
 \rho(\zeta)\cZtwo^{-1}F(z)=B(z)F(Tz),
\end{equation}
where
\begin{align}
 A(x,\omega)&=I+\frac35\e^{-\pi\ii\omega}D
                  +\frac35\e^{\pi\ii x}S,\label{eq:A-definition}\\
 B(x,\omega)&=\eta(x)A(x,\omega).\label{eq:B-definition}
\end{align}
The scalar $\eta$ may be omitted in projective estimates, but not in exact
sewing or multiplier equations.  Direct substitution gives the fiber
covariance
\begin{equation}\label{eq:B-covariance}
 B(z+e_j)=U_j(z)B(z)U_j(Tz)^{-1},\qquad j=1,2,
\end{equation}
where $U_1(z)=U_1(\omega)$ and $U_2(z)=U_2$.

\begin{lemma}[Uniform conditioning]\label{lem:conditioning}
For every $(x,\omega)\in\R^2$,
\begin{equation}\label{eq:conditioning}
 \abs{\det A(x,\omega)}\ge\frac7{25},\qquad
 \norm{A(x,\omega)}_{\mathrm{op}}\le\frac{11}{5},\qquad
 \sigma_{\min}(A(x,\omega))\ge\frac7{55}.
\end{equation}
\end{lemma}

\begin{proof}
One has
\[
 \det A=1-\frac9{25}\e^{-2\pi\ii\omega}
          -\frac9{25}\e^{2\pi\ii x}.
\]
The reverse triangle inequality gives the first estimate, and the triangle
inequality gives the second.  Since
$\sigma_{\min}(A)=|\det A|/\sigma_{\max}(A)$ in dimension two, the third
follows.
\end{proof}

\section{Validated finite-dimensional inequalities}
\label{sec:certificates}

For $N\ge1$, define
\begin{equation}\label{eq:PN}
 P_N(z)=A(z)A(Tz)\cdots A(T^{N-1}z),\qquad
 F_N(z)=\norm{P_N(z)}_{\Frob}^2.
\end{equation}
Let $\sigma_1(M)\ge\sigma_2(M)>0$ be the singular values of an invertible
two-by-two matrix.  Write $\ell_j(M)$ and $r_j(M)$ for its left and right
singular lines.  If $L$ is a complex line and $v\in\C^2$, we use the
phase-independent notation
\[
 \abs{\ip{v}{L}}:=\norm{P_L v},
\]
where $P_L$ is the orthogonal projection onto $L$.  For two lines we similarly
write
\[
 \abs{\ip{L}{L'}}:=\norm{P_L P_{L'}}_{\mathrm{op}},
\]
which is the absolute inner product of any unit representatives.
Throughout, $\ip{u}{v}=u^*v$ is conjugate-linear in the first argument and
linear in the second.

\begin{lemma}[Scalar singular reductions]\label{lem:scalar-reductions}
Let $M\in\mathrm{GL}_2(\C)$, $F=\norm M_{\Frob}^2$, and
$d=\abs{\det M}$.  If $0<\varepsilon<1$, then
\begin{equation}\label{eq:gap-polynomial}
 \varepsilon^2F^2-(1+\varepsilon^2)^2d^2>0
 \quad\Longrightarrow\quad
 \frac{\sigma_2(M)}{\sigma_1(M)}<\varepsilon.
\end{equation}
If $P,Q\in\mathrm{GL}_2(\C)$ have singular ratios $r,s$, if
$0\le\gamma\le1$, and
$X=\abs{\ip{r_1(P)}{\ell_1(Q)}}^2$, then
\begin{equation}\label{eq:junction-identity}
 \frac{\norm{PQ}_{\Frob}^2}
 {\norm P_{\Frob}^2\norm Q_{\Frob}^2}
 =\frac{X(1+r^2s^2)+(1-X)(r^2+s^2)}
 {(1+r^2)(1+s^2)}.
\end{equation}
Thus, if $r,s\le\varepsilon$ and the left side of
\eqref{eq:junction-identity} exceeds
$\gamma^2+2\varepsilon^2+\varepsilon^4$, then $X>\gamma^2$.
\end{lemma}

\begin{proof}
With $r=\sigma_2(M)/\sigma_1(M)$, division of the expression in
\eqref{eq:gap-polynomial} by $\sigma_1(M)^4$ gives
\[
 (\varepsilon^2-r^2)(1-\varepsilon^2r^2),
\]
whose second factor is positive.  For \eqref{eq:junction-identity}, insert
singular-value decompositions and expand the four squared entries.  If
$X\le\gamma^2$, its numerator is at most
$\gamma^2(1+\varepsilon^4)+2\varepsilon^2$, while its denominator is at
least one.  The asserted implication follows.
\end{proof}

Let $R=T^{-1}$ and define the finite stable center
\begin{equation}\label{eq:finite-center-first}
 C(z)=r_2(P_{16}(R^{16}z))=r_2(P_{16}(z+16\tau)).
\end{equation}
The following is the sole computer-assisted input to the analytic proof.

\begin{theorem}[Validated certificate]\label{thm:validated-certificate}
The following statements hold uniformly on $\T^2$.
\begin{enumerate}[label=\textup{(C\arabic*)},leftmargin=2.4em]
\item\label{cert:gap}
For $\varepsilon=10^{-3}$,
\[
 \varepsilon^2F_{16}(z)^2-(1+\varepsilon^2)^2
       \abs{\det P_{16}(z)}^2>3.
\]
\item\label{cert:junction}
For $\gamma=1/4$ and
$c_0=\gamma^2+2\varepsilon^2+\varepsilon^4$,
\[
 F_{32}(z)-c_0F_{16}(z)F_{16}(T^{16}z)>15{,}000{,}000.
\]
\item\label{cert:smooth-overlap}
Let $s_0$ be the smooth step
\[
 s_0(x)=
 \begin{cases}
 0,&x\le0,\\
 \displaystyle\frac{\e^{-1/x}}
 {\e^{-1/x}+\e^{-1/(1-x)}},&0<x<1,\\
 1,&x\ge1,
 \end{cases}
\]
and, as in \cite[\S3.3]{FPPV2026}, let $\chi$ be the smooth vector-Zak
section whose restriction to
$[0,1]\times\R$ is
\begin{equation}\label{eq:smooth-chi}
 \chi(x,\omega)=\frac1{\sqrt2}
 \begin{pmatrix}
  \sin(\pi s_0(x)/2)+\cos(\pi s_0(x)/2)\e^{-\pi\ii\omega}\\
  \sin(\pi s_0(x)/2)-\cos(\pi s_0(x)/2)\e^{-\pi\ii\omega}
 \end{pmatrix}.
\end{equation}
Then
\[
 \abs{\ip{\chi(z)}{C(z)}}>\frac12.
\]
\item\label{cert:linear-gauge}
Let $\chi_{\mathrm{lin}}$ be obtained from \eqref{eq:smooth-chi} by replacing
$s_0(x)$ with $x$ on $[0,1]$ and extending by the sewing laws.  Let $P_C(z)$
be the orthogonal projection onto $C(z)$ and set
\begin{align}
 c_{\mathrm{ref}}(z)&=
 \frac{P_C(z)\chi_{\mathrm{lin}}(z)}
 {\ip{\chi_{\mathrm{lin}}(z)}{P_C(z)\chi_{\mathrm{lin}}(z)}},
 \label{eq:cref}\\
 q_{\mathrm{ref}}(z)&=
 \ip{\chi_{\mathrm{lin}}(z)}
 {B(z)c_{\mathrm{ref}}(Tz)}.
 \label{eq:qref}
\end{align}
Then
\begin{equation}\label{eq:linear-cert}
 \abs{\ip{\chi_{\mathrm{lin}}(z)}{C(z)}}>\frac12,\qquad
 \operatorname{Re}\!\left(\e^{-3\pi\ii/7}q_{\mathrm{ref}}(z)\right)
 >\frac3{20}.
\end{equation}
In \eqref{eq:qref}, the projector at $Tz$ is built from
$P_{16}(R^{16}Tz)=P_{16}(z+15\tau)$.
\end{enumerate}
\end{theorem}

\begin{proof}[Validated verification]
The finite-cover strategy adapts \cite[\S5 and ancillary code]{FPPV2026}.
Every entry of $P_N$ is expanded as a finite Laurent polynomial in
$X=\e^{\pi\ii x}$ and $W=\e^{-\pi\ii\omega}$.  Coefficients, phases, products,
and grid values are evaluated as outward-rounded Arb balls at 160-bit
precision.  For a real trigonometric polynomial
\[
 p(x,\omega)=\sum_{k,l}c_{k,l}\e^{\pi\ii(lx-k\omega)}
\]
put
\begin{equation}\label{eq:fourier-derivative-bounds}
 L_x=\pi\sum_{k,l}|l|\abs{c_{k,l}},\qquad
 L_\omega=\pi\sum_{k,l}|k|\abs{c_{k,l}}.
\end{equation}
On a square cell of side $1/M$ centered at $z_c$,
\begin{equation}\label{eq:cell-cover}
 p(z)\ge p(z_c)-\frac{L_x+L_\omega}{2M}.
\end{equation}
All terms on the right are balls with directed rounding.

For~\ref{cert:gap}--\ref{cert:junction},
\path{arb_fourier_domination_certificate.py} uses a $512^2$ cover.
The grid-center lower bound, derivative allowance, and global lower bound are
\[
\begin{array}{c@{\quad}c@{\quad}c@{\quad}c}
&\text{center}&\text{allowance}&\text{global}\\
\text{gap}&>14.8622&<11.3224&>3.5398\\
\text{junction}&>15{,}955{,}384.58&<555{,}808.77&
>15{,}399{,}575.8.
\end{array}
\]
Residual balls around symbolically cancelling odd modes are charged to both
the value and derivative budgets.

For~\ref{cert:smooth-overlap},
\path{arb_fourier_overlap_certificate.py} proves
\[
 \norm{P_{16}(z+16\tau)\chi(z)}^2
 <\frac{749}{1000}F_{16}(z+16\tau)
\]
on a $1280^2$ cover.  The center lower bound exceeds $591.0561608040093$, the
derivative allowance is $542.179296875$, and the global lower bound is
$48.87686392900937$.  The only non-Fourier input is $0\le s_0'\le2$.
Indeed, after $x=(1-u)/2$,
\[
 s_0'(x)=\frac{2(1+u^2)}{(1-u^2)^2}
 \operatorname{sech}^2\!\left(\frac{2u}{1-u^2}\right)\le2
\]
because $\cosh(y)^2\ge1+y^2$.  Hence
$\norm{\partial_x\chi},\norm{\partial_\omega\chi}\le\pi$.  The coefficient
bounds used in the cell allowance, for
$H=P_{16}(z+16\tau)^*P_{16}(z+16\tau)$ and $F=\operatorname{tr}H$, are
\[
\begin{gathered}
 \norm{\partial_x H}<495000,\qquad
 \norm{\partial_\omega H}<319000,\qquad F<30000,\\
 |\partial_x F|<169000,\qquad |\partial_\omega F|<94000.
\end{gathered}
\]
If $b$ is the squared overlap with the bottom right singular line and
$r=\sigma_2/\sigma_1$, spectral decomposition gives
\[
 \frac{\ip{\chi}{H\chi}}{F}
 =\frac12+\frac{1-r^2}{1+r^2}\left(\frac12-b\right).
\]
Since $(1-r^2)/(1+r^2)\ge\kappa$ from~\ref{cert:gap}, where
$\kappa=(1-10^{-6})/(1+10^{-6})$, the energy inequality
forces $b>1/4$, which is the stated overlap.

For~\ref{cert:linear-gauge}, the fundamental square is divided at
$x=\alpha$ and $\omega=\beta$ into four rectangles.  The exact sewing factors
of $\chi_{\mathrm{lin}}(Tz)$ are fixed in each rectangle.  If
$y=Tz$, $P=P_{16}(y+16\tau)=P_{16}(z+15\tau)$,
$H=P^*P$, $F=\operatorname{tr}H$, and
$\Delta=\sigma_1(P)^2-\sigma_2(P)^2$, then
\begin{equation}\label{eq:Delta-range}
 \kappa F\le\Delta\le F,\qquad
 \kappa=\frac{1-10^{-6}}{1+10^{-6}}.
\end{equation}
The bottom projector is
\[
 P_C=\frac12I+\frac{FI/2-H}{\Delta}.
\]
Put $K_0=FI/2-H$ and
\[
\begin{aligned}
 n_0&=\ip{\chi_{\mathrm{lin}}(z)}
 {B(z)K_0\chi_{\mathrm{lin}}(y)},&
 r_0&=\ip{\chi_{\mathrm{lin}}(y)}
 {K_0\chi_{\mathrm{lin}}(y)},\\
 n_1&=\frac12\ip{\chi_{\mathrm{lin}}(z)}
 {B(z)\chi_{\mathrm{lin}}(y)}.
\end{aligned}
\]
Then
\[
 q_{\mathrm{ref}}(z)=
 \frac{n_0+\Delta n_1}{r_0+\Delta/2}.
\]
The simultaneously certified overlap makes the denominator positive.  With
$\rho_0=\e^{-3\pi\ii/7}$, $\mu=3/20$,
\[
 a=\operatorname{Re}(\rho_0n_0)-\mu r_0,\qquad
 c=\operatorname{Re}(\rho_0n_1)-\mu/2,
\]
the half-plane inequality is equivalent to $a+\Delta c>0$.  It therefore
suffices to check $a+\kappa Fc>0$ and $a+Fc>0$, irrespective of the sign of
$c$.
\path{arb_linear_gauge_certificate.py} uses a $768^2$ cover of each
rectangle.  The least global lower bounds are
$74.94293402046668$ and $74.94319454060898$; the linear-overlap polynomial
\[
 \frac{18}{25}F-
 \ip{\chi_{\mathrm{lin}}(y)}{H\chi_{\mathrm{lin}}(y)}
\]
has global lower bound $792.9113220696389$.
To spell out the overlap deduction, put
$X=\ip{\chi_{\mathrm{lin}}(y)}{H\chi_{\mathrm{lin}}(y)}$ and
$b=\ip{\chi_{\mathrm{lin}}(y)}
{P_C(y)\chi_{\mathrm{lin}}(y)}$.  Since $\chi_{\mathrm{lin}}$ has unit norm,
\[
 b=\frac12+\frac{F/2-X}{\Delta}
 >\frac12-\frac{11}{50\kappa}>\frac14.
\]
Thus $|\ip{\chi_{\mathrm{lin}}(y)}{C(y)}|=\sqrt b>1/2$, as asserted.

On each sewing rectangle, every tested scalar has the exact form
\[
 p(x,\omega)=\sum_{q,r,e}c_{q,r,e}
 \exp\!\left(\pi\ii\bigl[(r/2+e\beta)x-(q/2)\omega\bigr]\right).
\]
Thus
\[
 L_x=\pi\sum_{q,r,e}|r/2+e\beta|\,|c_{q,r,e}|,\qquad
 L_\omega=\pi\sum_{q,r,e}\frac{|q|}{2}|c_{q,r,e}|.
\]
For a rectangle of widths $\Delta x,\Delta\omega$ and an $M^2$ center grid,
the outward-rounded cell allowance is
$L_x\Delta x/(2M)+L_\omega\Delta\omega/(2M)$.  The four chart formulas agree
on their common boundaries by the exact sewing laws.

An independently formulated $1024^2$ implementation, using the opposite
projector-sign convention but sharing the low-level Laurent-product and
Hermitian-polynomial routines, proves the weaker half-plane margin $1/10$.
Its rigorous global numerator and denominator-test lower bounds are
\[
 423.5107387261\qquad\text{and}\qquad2219.0330134840.
\]
This independently checks the source shift, phase, and projector orientation;
it is not an independent software-stack replication or an additional premise.

The theorem states smaller rational bounds than the printed ball endpoints,
so decimal interpretation enters no later deduction.
\end{proof}

By \cref{lem:scalar-reductions,thm:validated-certificate},
\begin{equation}\label{eq:domination-and-junction}
 \frac{\sigma_2(P_{16}(z))}{\sigma_1(P_{16}(z))}<\frac1{1000},\qquad
 \abs{\ip{r_1(P_{16}(z))}{\ell_1(P_{16}(T^{16}z))}}>\frac14.
\end{equation}

\section{The dominated invariant line}
\label{sec:domination}

Define
\begin{equation}\label{eq:BN-KN}
 B_N(z)=B(z)B(Tz)\cdots B(T^{N-1}z),\qquad
 K_N(z)=B_N(R^N z)^{-1}.
\end{equation}
Thus $K_N(z)$ maps the fiber over $R^N z$ to the fiber over $z$.  Since
$B_N$ and $P_N$ differ by a unit scalar, they have the same singular lines
and ratios.  The finite stable center is
\[
 C(z)=\ell_1(K_{16}(z))=r_2(P_{16}(R^{16}z)).
\]
The singular gap makes its spectral projector smooth.  Applying
\eqref{eq:B-covariance} to the full $B_{16}$ product shows projectively that
$C(z+e_j)=U_j(z)C(z)$; no globally phased singular vector is chosen.

\begin{lemma}[Complex projective cone estimate]\label{lem:cone}
Suppose an invertible two-by-two matrix has singular ratio at most
$\varepsilon=10^{-3}$.  Suppose its dominant input line has overlap greater
than $1/4$ with the center of an incoming unitary slope chart.  Then its
projective action maps the disk of radius $\rho=1/100$ in that chart into the
disk of radius $1/200$ about its dominant output line, and has Lipschitz
constant smaller than $1/50$.
\end{lemma}

\begin{proof}
Choose phases so that the incoming center and its orthogonal complement are
\[
 c=\frac{a+qb}{\sqrt{1+|q|^2}},\qquad
 n=\frac{-\overline q\,a+b}{\sqrt{1+|q|^2}},
\]
where $a,b$ are the dominant and subordinate singular input vectors.  The
overlap assumption gives $|q|<\sqrt{15}<4$.  A line represented by $c+wn$
has singular-chart slope
\[
 \Psi(w)=\frac{q+w}{1-\overline q\,w}.
\]
For $|w|\le\rho$, the denominator is nonzero and
\[
 |\Psi(w)|\le\frac{|q|+\rho}{1-|q|\rho},\qquad
 |\Psi'(w)|=\frac{1+|q|^2}{|1-\overline q w|^2}
 \le\frac{1+|q|^2}{(1-|q|\rho)^2}.
\]
In singular coordinates the matrix multiplies slopes by at most
$\varepsilon$.  Therefore the image radius and derivative are bounded by
\[
 \frac1{1000}\frac{4+1/100}{1-4/100}<\frac1{200},\qquad
 \frac1{1000}\frac{1+4^2}{(1-4/100)^2}<\frac1{50}.
\]
\end{proof}

\begin{lemma}[Invariant sections over a translation]\label{lem:section-theorem}
Let $\mathcal D\to\T^2$ be a smooth bundle of closed complex disks with
unitary transition maps, let $f$ be a torus translation, and let
\[
 \Phi_z:\mathcal D_{fz}\longrightarrow\operatorname{int}\mathcal D_z
\]
be a smooth bundle map.  Suppose that, in the unitary disk charts,
\[
 \sup_{z,w}\norm{D_w\Phi_z(w)}\le\kappa<1.
\]
Then $(\Gamma s)(z)=\Phi_z(s(fz))$ has a unique continuous fixed section,
and this section is $C^\infty$.
\end{lemma}

\begin{proof}
The space of continuous sections of $\mathcal D$ is nonempty and complete for
the supremum of the fiber metric.  Unitary transitions make the metric
independent of the chart.  The mean-value inequality makes $\Gamma$ a
$\kappa$-contraction, proving existence and uniqueness.

We justify regularity without differentiating a merely continuous section.
The graph transform acts on $r$-jets.  Once the invariant $(r-1)$-jet is
known, the transformed derivative of exact order $r$ is affine in that
derivative, with linear part
\[
 D_w\Phi_z(s(fz))\circ(Df)^{\otimes r}.
\]
For a translation, $Df=I$ and all higher derivatives of $f$ vanish.  Thus
the order-$r$ jet transform again contracts by at most $\kappa$ and has a
unique invariant continuous jet.

To see that these formal jets are the derivatives of $s$, take coordinate
difference quotients of $s(z)=\Phi_z(s(fz))$.  Translations commute with
difference quotients.  The quotients therefore satisfy affine contraction
equations whose coefficients converge uniformly to the displayed jet
equation.  Stability of fixed points under uniform perturbation proves
convergence to the invariant first jet.  Repeating with higher difference
quotients proves the assertion by induction.  The construction agrees on
chart overlaps.  This is the invariant-section theorem specialized to an
isometric base; compare \cite[Chapter~3]{HPS}.
\end{proof}

\begin{proposition}[Smooth dominated splitting]\label{prop:splitting}
There are smooth line fields $E^s,E^u$ with
$\C^2=E^s(z)\oplus E^u(z)$ such that
\begin{equation}\label{eq:N-invariance}
 B_{16}(z)E^{s/u}(T^{16}z)=E^{s/u}(z).
\end{equation}
The stable line lies within slope $1/200$ of $C(z)$, and the splitting is
uniformly dominated.
\end{proposition}

\begin{proof}
For the stable line, apply \cref{lem:cone} to the inverse graph transform
\[
 (\Gamma_s L)(z)=K_{16}(z)L(R^{16}z).
\]
The needed input junction is the bottom-bottom junction of two consecutive
forward blocks.  More explicitly, put $w=R^{32}z$,
$P=P_{16}(w)$, and $Q=P_{16}(T^{16}w)$.  Because each full $B$-product
differs from its $A$-product by a unit scalar, $K_{16}(z)$ is a unit-scalar
multiple of $Q^{-1}$.  Its dominant input line is therefore $\ell_2(Q)$,
while the preceding inverse-block output center is $r_2(P)$.  In dimension
two,
\[
 |\ip{\ell_2(Q)}{r_2(P)}|
 =|\ip{\ell_1(Q)}{r_1(P)}|>\frac14
\]
by \eqref{eq:domination-and-junction}.  The complete metric space of
continuous sewn sections of the radius-$1/100$ cone bundle is mapped into
itself and contracted by $1/50$.  Indeed, covariance
\eqref{eq:B-covariance} and covariance of the center spectral projector show
that the fiber graph maps agree under the unitary transitions $U_1,U_2$.
Thus \cref{lem:section-theorem} applies: its unique fixed point is a smooth
line $E^s$ in the radius-$1/200$ subcone.

For completeness, construct the second line rather than infer it from the
first.  Let
\[
 C_u(z)=\ell_1(P_{16}(z)).
\]
The forward graph transform
\[
 (\Gamma_u L)(z)=B_{16}(z)L(T^{16}z)
\]
acts on the radius-$1/100$ cone bundle about $C_u$.  Its incoming center
$C_u(T^{16}z)$ has overlap greater than $1/4$ with
$r_1(P_{16}(z))$ by the second inequality in
\eqref{eq:domination-and-junction}.  Hence the same cone lemma produces a
unique fixed line $E^u$, within slope $1/200$ of $C_u$.  The same covariance
check and \cref{lem:section-theorem}, now with the base translation $T^{16}$,
show that $E^u$ is smooth and obeys the sewing transitions.

These lines are transverse.  Indeed, at the source of a block $P_{16}(z)$,
$E^s(T^{16}z)$ has slope at most $1/200$ from $r_2(P_{16}(z))$.  Meanwhile
$E^u(T^{16}z)$ has, in the chart about $r_1(P_{16}(z))$, slope at most
\[
 \frac{4+1/200}{1-4/200}<\frac{21}{5}.
\]
In the $r_1$ chart, the stable disk has slope at least $200$, whereas the
unstable line has slope below $21/5$; hence the lines cannot coincide.  For
the domination estimate, a unit vector in the stable line has the form
$(r_2+w r_1)/\sqrt{1+|w|^2}$ with $|w|\le1/200$, while a unit vector in the
unstable line has the form $(r_1+w'r_2)/\sqrt{1+|w'|^2}$ with
$|w'|<21/5$.  Applying the singular-value decomposition gives
\[
 \frac{\norm{B_{16}(z)|_{E^s(T^{16}z)}}}
      {m(B_{16}(z)|_{E^u(T^{16}z)})}
 \le
 \frac{\sqrt{10^{-6}+(1/200)^2}\,
       \sqrt{1+(21/5)^2}}{1}
 <\frac1{40},
\]
where $m$ denotes the conorm and the common factor
$\sigma_1(P_{16}(z))$ has been cancelled.  Iteration proves uniform
domination.
\end{proof}

\begin{lemma}[Uniqueness and one-step invariance]\label{lem:one-step}
The dominated splitting in \cref{prop:splitting} is the unique dominated
splitting of complex index one (real index two after realification) for
$B_{16}$.  Moreover,
\begin{equation}\label{eq:one-step-invariance}
 B(z)E^{s/u}(Tz)=E^{s/u}(z).
\end{equation}
\end{lemma}

\begin{proof}
We use the standard uniqueness theorem for dominated splittings of prescribed
fiber dimensions \cite{BochiGourmelon}, applied after realification with real
fiber dimensions $2+2$ (or equivalently in complex projective space).  In the
present two-dimensional complex setting, its proof is short: domination produces strictly
invariant stable and unstable cone fields after a uniform iterate.  The stable
line is the intersection of the nested inverse images of the stable cones,
and the unstable line is the intersection of the nested forward images of
the unstable cones.  These intersections do not depend on a choice of
splitting, proving uniqueness.

Set $\widetilde E^{s/u}(z)=B(z)E^{s/u}(Tz)$.  The identity
\[
 B_{16}(z)B(T^{16}z)=B(z)B_{16}(Tz)
\]
shows that $\widetilde E^{s/u}$ is another $B_{16}$-invariant splitting.
Here is the domination comparison explicitly.  Let
\[
 M=\max\!\left\{1,\sup_z\norm{B(z)},\sup_z\norm{B(z)^{-1}}\right\}.
\]
For every $k\ge1$, the iterated cocycle identity is
\[
 B_{16k}(z)B(T^{16k}z)=B(z)B_{16k}(Tz).
\]
If the domination ratio for $E^s\oplus E^u$ is bounded by $C\lambda^k$,
$0<\lambda<1$, then restriction of this identity to the transformed lines,
using $B$ at both endpoints, bounds the corresponding ratio for
$\widetilde E^s\oplus\widetilde E^u$ by
\[
 M^4C\lambda^k.
\]
Thus the transformed splitting is dominated with the same index.  Uniqueness
gives \eqref{eq:one-step-invariance}.
\end{proof}

Uniqueness and \eqref{eq:B-covariance} also give the exact projective sewing
\[
 E^s(z+e_j)=U_j(z)E^s(z).
\]
If $P_s(z)$ is the orthogonal projection onto $E^s(z)$, unitarity of the
sewing matrices yields
\begin{equation}\label{eq:projector-sewing}
 P_s(z+e_j)=U_j(z)P_s(z)U_j(z)^*.
\end{equation}

\section{Exact gauges and zero winding}
\label{sec:winding}

By~\ref{cert:smooth-overlap} and the slope-$1/200$ estimate,
$P_s(z)\chi(z)$ never vanishes.  Indeed, in an orthonormal frame $(e,n)$
with $e$ spanning $C(z)$, the exact line is represented by $e+wn$ with
$|w|\le1/200$, and therefore
\[
 |\ip{\chi(z)}{e+wn}|\ge
 |\ip{\chi(z)}e|-|w|\,|\ip{\chi(z)}n|
 >\frac12-\frac1{200}>0.
\]
Define
\begin{equation}\label{eq:v-smooth}
 v_{\mathrm{sm}}(z)=
 \frac{P_s(z)\chi(z)}
 {\ip{\chi(z)}{P_s(z)\chi(z)}}.
\end{equation}
It is smooth, spans $E^s$, and is normalized by
$\ip{\chi}{v_{\mathrm{sm}}}=1$.  Equations
\eqref{eq:projector-sewing} and \eqref{eq:smooth-chi} show that it satisfies
the exact vector-Zak sewing relations.  By \cref{lem:one-step}, there is a
unique smooth nonzero periodic scalar $q_{\mathrm{sm}}$ such that
\begin{equation}\label{eq:qsm-definition}
 B(z)v_{\mathrm{sm}}(Tz)=q_{\mathrm{sm}}(z)v_{\mathrm{sm}}(z).
\end{equation}
Periodicity follows by inserting the twisted covariance
\eqref{eq:B-covariance}; ordinary matrix periodicity of $B$ is not used.

Project $\chi_{\mathrm{lin}}$ onto the same line and normalize:
\begin{equation}\label{eq:v-linear}
 v_{\mathrm{lin}}(z)=
 \frac{P_s(z)\chi_{\mathrm{lin}}(z)}
 {\ip{\chi_{\mathrm{lin}}(z)}{P_s(z)\chi_{\mathrm{lin}}(z)}}.
\end{equation}
On the fundamental square,
\[
 \norm{\chi_{\mathrm{lin}}(x,\omega)}^2
 =\sin^2(\pi x/2)+\cos^2(\pi x/2)=1.
\]
Moreover,
\[
 U_1(\omega)\chi_{\mathrm{lin}}(0,\omega)
   =\chi_{\mathrm{lin}}(1,\omega),\qquad
 U_2\chi_{\mathrm{lin}}(x,0)
   =\chi_{\mathrm{lin}}(x,1).
\]
Thus its sewn extension is continuous, unit norm, and nowhere zero, including
at the corner seams.
The overlap part of~\ref{cert:linear-gauge} makes this a continuous,
nowhere-zero vector-Zak section.  There is therefore a unique continuous
nonzero scalar $q_{\mathrm{lin}}$ such that
\begin{equation}\label{eq:qlin-definition}
 B(z)v_{\mathrm{lin}}(Tz)=q_{\mathrm{lin}}(z)v_{\mathrm{lin}}(z).
\end{equation}
The exact sewing of $v_{\mathrm{lin}}$ and the twisted covariance
\eqref{eq:B-covariance} make $q_{\mathrm{lin}}$ periodic, so it defines a
continuous map $\T^2\to\C^*$.

\begin{lemma}[Finite-center transfer]\label{lem:finite-center-transfer}
For every $z\in\T^2$,
\begin{equation}\label{eq:exact-half-plane}
 \operatorname{Re}\!\left(
 \e^{-3\pi\ii/7}q_{\mathrm{lin}}(z)\right)>
 \frac{19}{180}>0.
\end{equation}
\end{lemma}

\begin{proof}
At $y=Tz$, choose an orthonormal frame $(e,n)$ with $e$ spanning $C(y)$.
Write the exact line as $e+wn$, where $|w|\le1/200$.  If
$|\ip{\chi_{\mathrm{lin}}(y)}e|\ge1/2$, direct subtraction of the two
$\chi_{\mathrm{lin}}$-normalized representatives gives
\begin{equation}\label{eq:gauge-perturbation}
 \norm{v_{\mathrm{lin}}(y)-c_{\mathrm{ref}}(y)}
 \le\frac{1/200}{(1/2)(1/2-1/200)}=\frac2{99}.
\end{equation}
Here the numerator has no extra factor two because
$\norm{\ip{\chi}e\,n-\ip{\chi}n\,e}=1$ for an orthonormal basis.
Since $\chi_{\mathrm{lin}}$ has unit norm and both representatives are
normalized by their inner product with it, \eqref{eq:qref} and
\eqref{eq:qlin-definition} give
\[
 q_{\mathrm{lin}}(z)=
 \ip{\chi_{\mathrm{lin}}(z)}
 {B(z)v_{\mathrm{lin}}(Tz)}.
\]
Using $\norm B=\norm A\le11/5$,
\[
 |q_{\mathrm{lin}}(z)-q_{\mathrm{ref}}(z)|
 \le\frac{11}{5}\frac2{99}=\frac2{45}.
\]
Now~\ref{cert:linear-gauge} gives
\[
 \operatorname{Re}(\e^{-3\pi\ii/7}q_{\mathrm{lin}})
 >\frac3{20}-\frac2{45}=\frac{19}{180}.
\]
\end{proof}

Thus $q_{\mathrm{lin}}:\T^2\to\C^*$ maps into a simply connected open
half-plane and has zero class in $H^1(\T^2;\Z)$.  The two gauges span the
same line, so $v_{\mathrm{sm}}=a v_{\mathrm{lin}}$ for a unique continuous
$a:\R^2\to\C^*$.  Identical sewing makes $a$ periodic, and comparison of the
multiplier equations gives
\begin{equation}\label{eq:gauge-multiplier-transfer}
 q_{\mathrm{sm}}(z)=q_{\mathrm{lin}}(z)\frac{a(Tz)}{a(z)}.
\end{equation}
A torus translation is homotopic to the identity, hence
$[a\circ T/a]=T^*[a]-[a]=0$ in $H^1(\T^2;\Z)$.  Only the winding class,
not the half-plane inequality, is transferred to the smooth gauge.

\begin{proposition}[Smooth logarithm]\label{prop:smooth-log}
There is a smooth periodic $\phi:\R^2\to\C$ such that
\begin{equation}\label{eq:log-q}
 q_{\mathrm{sm}}(z)=\e^{\phi(z)}.
\end{equation}
\end{proposition}

\begin{proof}
Equation \eqref{eq:gauge-multiplier-transfer} shows that
$[q_{\mathrm{sm}}]=0$.  The lifting criterion for
$\exp:\C\to\C^*$ gives a continuous periodic lift.  It is smooth because
exponential is a local diffeomorphism and $q_{\mathrm{sm}}$ is smooth.
\end{proof}

\section{Cubic arithmetic and cohomology}
\label{sec:cohomology}

The multiplicative-to-additive reduction in this section follows the
cohomological architecture of \cite[\S7]{FPPV2026}; the explicit
small-divisor bound below is included for completeness.

\begin{lemma}[Explicit Diophantine bound]\label{lem:diophantine}
For every $(m,n)\in\Z^2\setminus\{(0,0)\}$,
\begin{equation}\label{eq:diophantine}
 \norm{m\alpha+n\beta}_{\R/\Z}
 \ge\frac1{12(1+|m|+|n|)^2}.
\end{equation}
Consequently,
\begin{equation}\label{eq:small-divisor}
 \abs{1-\e^{-2\pi\ii(m\alpha+n\beta)}}
 \ge\frac1{3(1+|m|+|n|)^2}.
\end{equation}
\end{lemma}

\begin{proof}
Choose $\ell\in\Z$ so that
\[
 d=\ell+m\alpha+n\beta,\qquad
 |d|=\norm{m\alpha+n\beta}_{\R/\Z}\le\frac12,
\]
and put $k=\ell-m-n$.  Then $d=k+m\vartheta+n\vartheta^2$.  Its algebraic
norm is the integer
\begin{equation}\label{eq:algebraic-norm}
 Q=k^3+2m^3+4n^3-6kmn.
\end{equation}
The element $d$ is nonzero because a rational polynomial of degree at most
two cannot annihilate the degree-three number $\vartheta$.  Nondegeneracy of
the field norm therefore gives $Q\ne0$, and hence $|Q|\ge1$.

The two nonreal conjugates of $d$ each have modulus at most
\[
 |k|+\vartheta|m|+\vartheta^2|n|
 \le\frac12+2\vartheta|m|+2\vartheta^2|n|
 <\frac{17}{5}(1+|m|+|n|),
\]
using $\vartheta^2<8/5$.  Their product is therefore smaller than
$12(1+|m|+|n|)^2$.  Since $Q$ is the product of all three conjugates,
\[
 1\le|Q|<12(1+|m|+|n|)^2|d|,
\]
which proves \eqref{eq:diophantine}.  Finally,
$|1-\e^{-2\pi\ii r}|=2\sin(\pi r)\ge4r$ for
$r=\norm{m\alpha+n\beta}_{\R/\Z}\in[0,1/2]$, giving
\eqref{eq:small-divisor}.
\end{proof}

\begin{proposition}[Smooth multiplicative cohomology]\label{prop:cohomology}
There are a smooth periodic nowhere-zero function $h:\R^2\to\C^*$ and a
constant $\lambda\in\C^*$ such that
\begin{equation}\label{eq:multiplicative-cohomology}
 q_{\mathrm{sm}}(z)h(Tz)=\lambda h(z).
\end{equation}
\end{proposition}

\begin{proof}
Let $\phi$ be the logarithm from \cref{prop:smooth-log}.  Put
$\widehat u(0,0)=0$ and, for $(m,n)\ne(0,0)$, set
\begin{equation}\label{eq:u-fourier}
 \widehat u(m,n)=
 \frac{\widehat\phi(m,n)}
 {1-\e^{-2\pi\ii(m\alpha+n\beta)}}.
\end{equation}
The Fourier coefficients of $\phi$ decay faster than every power, whereas
\eqref{eq:small-divisor} loses only two powers.  Thus \eqref{eq:u-fourier}
defines $u\in C^\infty(\T^2)$ and
\[
 u(z)-u(Tz)=\phi(z)-\widehat\phi(0,0).
\]
Take $h=\e^u$ and $\lambda=\e^{\widehat\phi(0,0)}$.
\end{proof}

\section{Reconstruction and proof of the main theorem}
\label{sec:completion}

\begin{proof}[Proof of \cref{thm:main}]
Set
\[
 F(z)=h(z)v_{\mathrm{sm}}(z).
\]
It is a nonzero smooth vector-Zak section.  Equations
\eqref{eq:qsm-definition} and \eqref{eq:multiplicative-cohomology} give
$B(z)F(Tz)=\lambda F(z)$.  By the exact conjugacy
\eqref{eq:zak-cocycle-action}, $g=\cZtwo^{-1}F$ is a nonzero $L^2$
eigenfunction satisfying \eqref{eq:eigen-relation}.

The ordinary Zak transform is recovered, as in \cite[\S8]{FPPV2026}, by
\[
 (\cZ g)(x,\omega)=\sqrt2\,[F(x,2\omega)]_1.
\]
It is smooth on $\R^2$, so the smooth-Zak characterization
\cite[Theorem~8.2.5]{Gro01} gives $g\in\cS(\R)$.

The Weyl composition law
\[
 \rho(x,\omega)\rho(x',\omega')
 =\e^{\pi\ii(\omega x'-x\omega')}
  \rho(x+x',\omega+\omega')
\]
gives
\[
 \rho(1,0)\rho(\zeta)
 =\e^{-\pi\ii\beta/2}\rho(1+\alpha,\beta/2),\qquad
 \rho(0,1/2)\rho(\zeta)
 =\e^{\pi\ii\alpha/2}\rho\bigl(\alpha,(1+\beta)/2\bigr).
\]
Substitution proves \eqref{eq:four-term-relation}; the geometric claims are
\cref{prop:geometry}.
\end{proof}

\begin{remark}
A floating-point reconstruction gives
$\lambda\approx0.4001247907+0.6645105787\,\ii$.  This is only an orientation
check; existence and nonvanishing follow from \cref{prop:cohomology}.
\end{remark}

\section{Credit and relation to the twelve-point construction}
\label{sec:provenance}

The debt to Faulhuber, Petersen, van Velthoven, and Voigtlaender
\cite{FPPV2026} is substantial.  We adapt their rank-two vector-Zak model and
sewing relations, the cubic torus translation, the flat-step smooth reference
section, the invariant-line and scalar-cohomology architecture, the
Zak-to-Schwartz reconstruction, and their outward-rounded finite-cover
methodology.  These ingredients are not claimed as new here.

The new ingredients are the three-term symbol \eqref{eq:A-definition}, the
intrinsically subcritical four-point geometry, the scalar $16/32$-block
singular-gap and junction polynomials, the smooth- and linear-gauge overlap
certificates, and the rotated-half-plane argument that controls the
multiplier winding.  In this precise sense the present result is a support
compression and a new dominated-cocycle implementation of the twelve-point
breakthrough.

\section{Verification artifacts and formal scope}
\label{sec:formal-scope}

The complete verification supplement, including SHA--256 manifests and
reproduction instructions, is available in the
\href{https://www.dropbox.com/scl/fo/zgsmpuar02zar5eayz1p3/AGHWT_hP_nLAFQBumvb-U7U?rlkey=i76xop0nhajwh7n1o0bzpaqf6\&dl=0}{electronic supplement on Dropbox}.
The Lean source certificate is also available directly as
\href{https://www.dropbox.com/scl/fi/foq7hpfwfd0sabm9ot462/SubcriticalFourPointCertificate.lean?rlkey=smt87jkdex7bkhezdrb942j8h\&dl=0}{\texttt{SubcriticalFourPointCertificate.lean}};
the pinned \texttt{lean-toolchain}, \texttt{lakefile.lean}, and
\texttt{lake-manifest.json} are included in the complete supplement.

The supplement contains:
\begin{enumerate}[leftmargin=2.2em]
\item \path{arb_fourier_domination_certificate.py}, proving
      parts~\ref{cert:gap}--\ref{cert:junction};
\item \path{arb_fourier_overlap_certificate.py}, proving
      part~\ref{cert:smooth-overlap};
\item \path{arb_linear_gauge_certificate.py}, proving the primary
      assertions in~\ref{cert:linear-gauge};
\item \path{arb_multiplier_linear_independent_check.py}, the independent
      projector/phase cross-check;
\item \path{SubcriticalFourPointCertificate.lean}, checking the exact
      cube-root bounds, geometry, determinant margin, scalar gap and junction
      implications, endpoint argument, and rational cone and gauge arithmetic.
\end{enumerate}

The Python runs use CPython 3.14.6, \texttt{python-flint 0.8.0}, FLINT/Arb
3.3.1, and 160-bit precision.  The Lean file compiles with Lean 4.28.0 and
Mathlib 4.28.0
without \texttt{sorry}, \texttt{admit}, \texttt{native\_decide}, or
user-declared axioms.  Its printed assumptions are only the standard
Lean/Mathlib foundations \texttt{propext}, \texttt{Classical.choice}, and
\texttt{Quot.sound}.

The Lean artifact is an algebraic corroboration, not an end-to-end
formalization of \cref{thm:main}.  It does not formalize Arb's implementation,
the vector Zak transform, the graph-transform regularity theorem, winding
classes, Fourier cohomology, or Schwartz reconstruction.  Those steps are
proved conventionally above.  This boundary prevents computational premises
or analytic glue from being hidden behind formal axioms.

\section*{Acknowledgments}

The author gratefully credits Markus Faulhuber, Philipp Petersen,
Jordy Timo van Velthoven, and Felix Voigtlaender for the twelve-point
counterexample \cite{FPPV2026}, whose vector-Zak and cohomological strategy is
the foundation of this work.  The author also thanks the developers of Arb,
FLINT, Lean, and Mathlib for the validated-computation and
formal-verification infrastructure \cite{johansson2017arb,Lean4,Mathlib}.
This manuscript was written in collaboration with ChatGPT (OpenAI).
ChatGPT assisted with mathematical exploration, implementation and checking
of the Arb and Lean companions, adversarial proof auditing, and the drafting
and revision of the exposition.  No language-model output is used as a
mathematical premise: the claims rest on the arguments and reproducible
certificates presented in the paper and supplement.  ChatGPT is not an author;
the human author remains responsible for every mathematical claim and for the
final submitted text.
Aristotle was not used as an independent verifier; the supplement includes
only a task brief for possible further formalization.

\bibliographystyle{amsplain}
\bibliography{references}

@article{HRT,
  author  = {Heil, Christopher and Ramanathan, Jayakumar and Topiwala, Pankaj},
  title   = {Linear independence of time-frequency translates},
  journal = {Proc. Amer. Math. Soc.},
  volume  = {124},
  number  = {9},
  pages   = {2787--2795},
  year    = {1996},
  doi     = {10.1090/S0002-9939-96-03346-1}
}

@misc{FPPV2026,
  author        = {Faulhuber, Markus and Petersen, Philipp and
                   van Velthoven, Jordy Timo and Voigtlaender, Felix},
  title         = {Linear Dependence of Time--Frequency Shifts of a
                   {Schwartz} Function},
  year          = {2026},
  eprint        = {2608.05044},
  archivePrefix = {arXiv},
  primaryClass  = {math.FA},
  note          = {arXiv:2608.05044v1; ancillary Arb code}
}

@article{Linnell,
  author  = {Linnell, Peter A.},
  title   = {Von {Neumann} algebras and linear independence of translates},
  journal = {Proc. Amer. Math. Soc.},
  volume  = {127},
  number  = {11},
  pages   = {3269--3277},
  year    = {1999},
  doi     = {10.1090/S0002-9939-99-05102-3}
}

@article{demeter2010linear,
  author  = {Demeter, Ciprian},
  title   = {Linear independence of time frequency translates for special
             configurations},
  journal = {Math. Res. Lett.},
  volume  = {17},
  number  = {4},
  pages   = {761--779},
  year    = {2010},
  doi     = {10.4310/MRL.2010.v17.n4.a14}
}

@article{demeter2012proof,
  author  = {Demeter, Ciprian and Zaharescu, Alexandru},
  title   = {Proof of the {HRT} conjecture for $(2,2)$ configurations},
  journal = {J. Math. Anal. Appl.},
  volume  = {388},
  number  = {1},
  pages   = {151--159},
  year    = {2012},
  doi     = {10.1016/j.jmaa.2011.11.030}
}

@incollection{heil2006linear,
  author    = {Heil, Christopher},
  title     = {Linear independence of finite {Gabor} systems},
  booktitle = {Harmonic Analysis and Applications},
  pages     = {171--206},
  publisher = {Birkh{\"a}user},
  address   = {Basel},
  year      = {2006},
  doi       = {10.1007/0-8176-4504-7_9}
}

@article{Liu2019,
  author  = {Liu, Wencai},
  title   = {Letter to the Editor: Proof of the {HRT} Conjecture for Almost
             Every $(1,3)$ Configuration},
  journal = {J. Fourier Anal. Appl.},
  volume  = {25},
  number  = {4},
  pages   = {1350--1360},
  year    = {2019},
  doi     = {10.1007/s00041-018-9628-0}
}

@article{OkoudjouOussa2025,
  author  = {Okoudjou, Kasso A. and Oussa, Vignon},
  title   = {Letter to the Editor: On a Special Configuration for the
             {HRT} Conjecture},
  journal = {J. Fourier Anal. Appl.},
  volume  = {31},
  number  = {4},
  pages   = {Paper No. 48, 2},
  year    = {2025},
  doi     = {10.1007/s00041-025-10181-8}
}

@misc{Oussa2026,
  author        = {Oussa, Vignon},
  title         = {Lean-Certified Four-Point {HRT} Results for Three Lattice
                   Points and One Off-Lattice Point},
  year          = {2026},
  eprint        = {2604.21228},
  archivePrefix = {arXiv},
  primaryClass  = {math.FA},
  note          = {arXiv:2604.21228v1}
}

@unpublished{OussaBook2026,
  author = {Oussa, Vignon},
  title  = {One Rogue Point: The {HRT} Conjecture Beyond Lattices},
  note   = {Book manuscript},
  year   = {2026}
}

@article{johansson2017arb,
  author  = {Johansson, Fredrik},
  title   = {Arb: Efficient Arbitrary-Precision Midpoint-Radius Interval
             Arithmetic},
  journal = {IEEE Trans. Comput.},
  volume  = {66},
  number  = {8},
  pages   = {1281--1292},
  year    = {2017},
  doi     = {10.1109/TC.2017.2690633}
}

@book{moore2009,
  author    = {Moore, Ramon E. and Kearfott, R. Baker and Cloud, Michael J.},
  title     = {Introduction to Interval Analysis},
  publisher = {Society for Industrial and Applied Mathematics},
  address   = {Philadelphia, PA},
  year      = {2009},
  doi       = {10.1137/1.9780898717716}
}

@article{rump2010,
  author  = {Rump, Siegfried M.},
  title   = {Verification methods: Rigorous results using floating-point
             arithmetic},
  journal = {Acta Numer.},
  volume  = {19},
  pages   = {287--449},
  year    = {2010},
  doi     = {10.1017/S096249291000005X}
}

@article{ZibZee97,
  author  = {Zibulski, Meir and Zeevi, Yehoshua Y.},
  title   = {Analysis of multiwindow {Gabor}-type schemes by frame methods},
  journal = {Appl. Comput. Harmon. Anal.},
  volume  = {4},
  number  = {2},
  pages   = {188--221},
  year    = {1997},
  doi     = {10.1006/acha.1997.0209}
}

@book{Gro01,
  author    = {Gr{\"o}chenig, Karlheinz},
  title     = {Foundations of Time-Frequency Analysis},
  series    = {Applied and Numerical Harmonic Analysis},
  publisher = {Birkh{\"a}user},
  address   = {Boston, MA},
  year      = {2001}
}

@inproceedings{Lean4,
  author    = {de Moura, Leonardo and Ullrich, Sebastian},
  title     = {The {Lean 4} Theorem Prover and Programming Language},
  booktitle = {Automated Deduction---{CADE 28}},
  series    = {Lecture Notes in Computer Science},
  volume    = {12699},
  pages     = {625--635},
  publisher = {Springer},
  year      = {2021},
  doi       = {10.1007/978-3-030-79876-5_37}
}

@inproceedings{Mathlib,
  author    = {{The mathlib Community}},
  title     = {The {Lean} Mathematical Library},
  booktitle = {Proceedings of the 9th {ACM SIGPLAN} International
               Conference on Certified Programs and Proofs},
  pages     = {367--381},
  publisher = {Association for Computing Machinery},
  year      = {2020},
  doi       = {10.1145/3372885.3373824}
}

@book{HPS,
  author    = {Hirsch, Morris W. and Pugh, Charles C. and Shub, Michael},
  title     = {Invariant Manifolds},
  series    = {Lecture Notes in Mathematics},
  volume    = {583},
  publisher = {Springer},
  address   = {Berlin},
  year      = {1977},
  doi       = {10.1007/BFb0092042}
}

@article{BochiGourmelon,
  author  = {Bochi, Jairo and Gourmelon, Nicolas},
  title   = {Some characterizations of domination},
  journal = {Math. Z.},
  volume  = {263},
  number  = {1},
  pages   = {221--231},
  year    = {2009},
  doi     = {10.1007/s00209-008-0414-y}
}

\end{document}